\documentclass[12pt,reqno]{amsart}

\usepackage{amsmath, amssymb, amsthm}
\usepackage{enumerate}
\usepackage{mathtools}
\usepackage{cleveref}
\usepackage{color}
\usepackage{setspace}
\usepackage{tikz-cd}
\usepackage{xparse}
\usepackage{braket}
\usepackage{float}
\usepackage{mathrsfs}
\usepackage{graphicx}

\DeclarePairedDelimiter\abs{\lvert}{\rvert}%
\DeclarePairedDelimiter\norm{\lVert}{\rVert}%

\makeatletter
\let\oldabs\abs
\def\abs{\@ifstar{\oldabs}{\oldabs*}}
\let\oldnorm\norm
\def\norm{\@ifstar{\oldnorm}{\oldnorm*}}
\makeatother

\newcommand{\eps}{\varepsilon}
\newcommand{\mc}[1]{\mathcal{#1}}
\newcommand{\mf}[1]{\mathfrak{#1}}
\newcommand{\bb}[1]{\mathbb{#1}}

\newcommand{\lp}{\left(}
\newcommand{\rp}{\right)}

\theoremstyle{definition}

\theoremstyle{plain}
\newtheorem{thm}{Theorem}
\newtheorem{lem}[thm]{Lemma}

\newtheorem{cor}[thm]{Corollary}

\newtheorem*{claim*}{Claim}
\newtheorem*{example*}{Example}
\newtheorem*{remark*}{Remark}

\newcommand{\Frob}{\operatorname{Frob}}
\newcommand{\GL}{\operatorname{GL}}
\newcommand{\SL}{\operatorname{SL}}

\newcommand{\Gal}{\operatorname{Gal}}
\newcommand{\Aut}{\operatorname{Aut}}
\newcommand{\rhobar}{\operatorname{\overline{\rho}}}

\newcommand{\vol}{\operatorname{vol}}

=650pt \headheight = 8pt \headsep = 10pt

\title{A lower bound on the proportion of modular elliptic curves over Galois CM fields}
\author{Zachary Feng}
\date{}

\begin{document}

\begin{abstract}
	We calculate an explicit lower bound on the proportion of elliptic curves that are modular over any Galois CM field not containing \( \zeta_5 \). Applied to imaginary quadratic fields, this proportion is at least \( 2/5 \). Applied to cyclotomic fields \( \bb{Q}(\zeta_n) \) with \( 5\nmid n \), this proportion is at least \( 1-\eps \) with only finitely many exceptions of \( n \), for any choice of \( \eps > 0 \).
\end{abstract}

\maketitle

\section{Introduction}

Let \( K \) be a number field. For \( X > 0 \), denote \( \mc{O}_{K,X} \) to be the set of \( \alpha\in\mc{O}_K \) such that for every embedding \( \sigma:K\to \bb{C} \) one has \( \abs{\sigma(\alpha)} < X \). We write \( \mc{E}_X \) for the set of pairs \( (A,B)\in\mc{O}_{K,X^4}\times\mc{O}_{K,X^6} \) such that \( \Delta(A,B) = -16(4A^3+27B^2)\neq 0 \). For \( (A,B)\in \mc{E}_X \), we write \( E_{A,B} \) for the elliptic curve over \( K \) given by the equation \( y^2 = x^3 + Ax+B \). We say that an elliptic curve \( E \) over \( K \) is \textit{modular} if either \( E \) has complex multiplication, or there exists a cuspidal regular algebraic automorphic representation \( \pi \) of \( \GL_2(\bb{A}_K) \) such that \( E \) and \( \pi \) have the same \( L \)-function.

The purpose of this article is to prove the following theorem.
\begin{thm}\label{thm:main}
	Let \( K \) be a Galois CM field with \( \zeta_5\notin K \). Then
	\[ M_K := \liminf_{X\to\infty} \frac{\abs{\{ (A,B)\in \mc{E}_X : E_{A,B} \text{ is modular} \}}}{\abs{\mc{E}_X}} \geq \lp 1-\frac{1}{5^f} \rp^{2r} \]
	where \( 5\mc{O}_K = (\mf{p}_1\cdots\mf{p}_r)^e \) with \( \mf{p}_i\neq \mf{p}_j \) and \( e,f \) are the ramification and inertial indices of \( 5 \) in \( K \).
\end{thm}
This theorem calculates an explicit value for the lower bound, thus offering an improvement, at least in the case of a Galois extension, to the original result by Allen, Khare, and Thorne in Theorem 10.1 of their article \cite{AKT19} which showed that \( M_K > 0 \) for every CM field \( K \) with \( \zeta_5\notin K \).

\section{Preliminaries}

Let \( K \) be a number field. Let \( E \) be an elliptic curve over \( K \). Let \( p \) be a prime. Let \( G_K := \Gal(\overline{K}/K) \). Let \( \rhobar_{E,p}: G_K\to \Aut(E[p]) \cong \GL_2(\bb{F}_p) \) be the Galois representation of \( G_K \) acting on the \( p \)-torsion points of \( E \) in \( \overline{K} \). Let \( \rhobar: G_K \to \GL_2(\overline{\bb{F}}_p) \) be any continuous representation. We say that a prime \( l\neq p \) is \textit{decomposed generic for \( \rhobar \)} if it splits completely in \( K \) and for any \( v\mid l \) in \( K \), \( \rhobar \) is unramified at \( v \) and the eigenvalues \( \alpha_v,\beta_v \) of \( \rhobar(\Frob_v) \) satisfy \( \alpha_v\beta_v^{-1}\notin \{ 1,l,l^{-1} \} \). We say that \( \rhobar \) is \textit{decomposed generic} if there is a prime \( l\neq p \) that is decomposed generic for \( \rhobar \). The original definition traces back to \cite{CS17}.

Our main tools will be the following lemmas.
\begin{lem}[Corollary 9.13 \cite{AKT19}]\label{lem:akt19}
	Let \( K \) be a CM field and let \( E \) be an elliptic curve over \( K \) satisfying the following conditions:
	\begin{enumerate}[(1)]
		\item \( \rhobar_{E,5}|_{G_{K(\zeta_5)}} \) is absolutely irreducible. If \( \rhobar_{E,5}(G_{K(\zeta_5)}) = \SL_2(\bb{F}_5) \), then \( \zeta_5\notin K \).
		\item For each place \( \mf{p}\mid 5 \) of \( K \), \( E_{K_\mf{p}} \) is ordinary (i.e. has good ordinary reduction or potentially multiplicative reduction).
		\item \( \rhobar_{E,5} \) is decomposed generic.
	\end{enumerate}
	Then \( E \) is modular.
\end{lem}
\begin{lem}[Lemma 2.3 \cite{AN20}]\label{lem:an20}
	Let \( K/\bb{Q} \) be a finite Galois extension and let \( \rhobar: G_K\to \GL_2(\overline{\bb{F}}_l) \) be a continuous representation with \( l > 3 \). If \( \rhobar(G_K)\supset \SL_2(\bb{F}_l) \), then \( \rhobar \) is decomposed generic.
\end{lem}
Fix a norm \( \norm{-} \) on \( \bb{R}\otimes_\bb{Z}\mc{O}_K^2 =\bb{R}^{2[K:\bb{Q}]} \). For \( X > 0 \) and an integer \( m \), define the sets
\begin{align*}
	B_K(X) &:= \{ (A,B)\in\mc{O}_K^2 : \Delta(A,B) \neq 0, \norm{(A,B)} \leq X \} \\
	B_{K,m}(X) &:= \{ (A,B)\in B_K(X) : \rhobar_{E_{A,B},m}(G_K)\not\supset \SL_2(\bb{Z}/m\bb{Z}) \}
\end{align*}
\begin{lem}[Proposition 5.7 \cite{Z10}]\label{lem:z10}
	For a positive integer \( m \),
	\[ \frac{\abs{B_{K,m}(X)}}{\abs{B_K(X)}} \ll_{K,\norm{-},m} \frac{\log(X)}{X^{[K:\bb{Q}]/2}} \]
\end{lem}
\begin{lem}[Theorem 4.1 \cite{S09}]\label{lem:s09}
	Let \( q \) be a power of \( p \). Let \( \bb{F}_q \) be a finite field of characteristic \( p \geq 3 \). Let \( E/\bb{F}_q \) be an elliptic curve given by a Weierstrass equation \[ E:y^2 = f(x) \] where \( f(x)\in\bb{F}_q[x] \) is a cubic polynomial with distinct roots in \( \overline{\bb{F}}_q \). Then \( E \) is supersingular if and only if the coefficient of \( x^{p-1} \) in \( f(x)^{(p-1)/2} \) is zero.
\end{lem}

\section{Proof of the main result}

\begin{proof}[Proof of \Cref{thm:main}]
	First, suppose \( \rhobar_{E,5}(G_{K(\zeta_5)})\supset \SL_2(\bb{F}_5) \). Then \( \rhobar_{E,5}|_{G_{K(\zeta_5)}} \) is absolutely irreducible. Let \( L/K \) be any abelian extension, and suppose \( \rhobar_{E,5}(G_K)\supset \SL_2(\bb{F}_5) \). There is a surjective map
	\[ \Gal(L/K)\twoheadrightarrow \rhobar_{E,5}(G_K)/\rhobar_{E,5}(G_L) \]
	induced by \( \rhobar_{E,5} \). Since \( \Gal(L/K) \) is abelian, and \( \SL_2(\bb{F}_5) \) is perfect, it follows that \( \SL_2(\bb{F}_5)\subset \rhobar_{E,5}(G_L) \). By letting \( L = K(\zeta_5) \), we can conclude that \( \rhobar_{E,5}(G_K)\supset \SL_2(\bb{F}_5) \) if and only if \( \rhobar_{E,5}(G_{K(\zeta_5)})\supset \SL_2(\bb{F}_5) \). 
	
	For \( x\in\mc{O}_K \), define \( \abs{x}_\infty := \max_\sigma \abs{\sigma x} \) where \( \sigma \) varies over all embeddings \( \sigma: K\to \bb{C} \) and \( \abs{-} \) is the complex absolute value. We want to extend \( \abs{-}_\infty \) to \( \bb{R}\otimes_{\bb{Z}}\mc{O}_K \) as a norm. Consider the embedding
	\begin{align*}
		\lambda: \mc{O}_K &\to \bb{R}^r\times\bb{C}^{s} \cong \bb{R}^{n} \\
		x &\mapsto (\sigma_1x,\dots,\sigma_rx,\tau_1x,\dots, \tau_sx)
	\end{align*}
	where \( \sigma_1,\dots,\sigma_r \) are the real embeddings and \( \tau_1,\dots,\tau_s \) are the non-real embeddings, and \( n = [K:\bb{Q}] \). By choosing the norm \( \norm{(x_1,\dots,x_r,z_1,\dots,z_s)}_\lambda := \max\{ \abs{x_1},\dots,\abs{x_r},\abs{z_1},\dots,\abs{z_s} \} \) on \( \bb{R}^n \), one has that \( \norm{x}_\infty = \norm{\lambda(x)}_\lambda \) for all \( x\in \mc{O}_K \). Since \( \lambda(\mc{O}_K) \) is a lattice in \( \bb{R}^n \), the map \( \bb{R}\otimes_\bb{Z}\mc{O}_K \to \bb{R}^n \) which sends \( t\otimes x \mapsto t\cdot \lambda(x) \) is an isomorphism of \( \bb{R} \)-vector spaces. By pulling back the norm \( \norm{-}_\lambda \) from \( \bb{R}^n \) to \( \bb{R}\otimes_\bb{Z}\mc{O}_K \) along this isomorphism, and calling this norm \( \norm{-}_\infty \) on \( \bb{R}\otimes_\bb{Z}\mc{O}_K \), we obtain that \( \norm{1\otimes x}_\infty = \norm{\lambda(x)}_\lambda = \norm{x}_\infty \) and hence we have extended \( \abs{-}_\infty \) to \( \norm{-}_\infty \) on \( \bb{R}\otimes_\bb{Z}\mc{O}_K \). Finally, we define a norm \( \norm{-} \) on the product \( \bb{R}\otimes_\bb{Z}\mc{O}_K^2 = (\bb{R}\otimes_\bb{Z}\mc{O}_K)^2 \) by setting \( \norm{(A,B)} := \max\{ \norm{A}_\infty,\norm{B}_\infty \} \). This is the norm we will use with \Cref{lem:z10}. Let us define two more sets.
	\begin{align*}
		\mc{E}_X' &= \{ (A,B)\in \mc{O}_{K,X^4}\times \mc{O}_{K,X^6} \} \\
		B_K'(X) &= \{ (A,B)\in \mc{O}_K^2 : \norm{(A,B)}\leq X \}
	\end{align*}
	Let \( \mc{C}_X = \{ A\in \bb{R}^n : \norm{A}_{\lambda} \leq X \} \) be the closed ball of radius \( X \). Indeed, \( \mc{C}_X = X\mc{C}_1 \). Since the boundary of \( \mc{C}_1 \) is \( (n-1) \)-Lipschitz parameterizable, Lemma 2 from Chapter 6 of \cite{M18} tells us that
	\[ \abs{\lambda(\mc{O}_K)\cap \mc{C}_X} = \abs{\lambda(\mc{O}_K)\cap X\mc{C}_1} =  \frac{\vol(\mc{C}_1)}{\vol(\bb{R}^n/\lambda(\mc{O}_K))}X^n + O(X^{n-1}) \]
	Therefore, there are constants \( \kappa_1,\kappa_2 > 0 \) such that
	\begin{align*}
		\abs{\mc{E}_X'} \sim \kappa_1X^{4n}X^{6n} \\
		\abs{B_K'(X^6)} \sim \kappa_2X^{6n}X^{6n}
	\end{align*}
	For each \( A\in \mc{O}_K \), \( \Delta(A,B) = 0 \) is a quadratic equation in \( B \) satisfied by at most two values of \( \mc{O}_K \). Therefore, there are constants \( \kappa_3,\kappa_4 > 0 \) such that
	\begin{align*}
		\kappa_1X^{4n}X^{6n} - \kappa_3X^{4n} \leq \abs{\mc{E}_X} &\leq \abs{\mc{E}_X'} \sim \kappa_1X^{4n}X^{6n} \\
		\kappa_2X^{6n}X^{6n} - \kappa_4X^{6n}\leq \abs{B_K(X^6)} &\leq \abs{B_K'(X^6)} \sim \kappa_2X^{6n}X^{6n}
	\end{align*}
	Therefore, \( \abs{\mc{E}_X} \sim \kappa_1 X^{4n}X^{6n} \) and \( \abs{B_K(X^6)} \sim \kappa_2 X^{6n}X^{6n} \). By combining our calculations with the estimate given in \Cref{lem:z10}, we can obtain the following asymptotic bound.
	\begin{align*}
		\frac{\abs{\left\{ (A,B)\in \mc{E}_X : \rhobar_{E_{A,B},5}(G_K)\not\supset \SL_2(\bb{F}_5) \right\}}}{\abs{\mc{E}_X}} &\ll \abs{B_{K,5}(X^6)} \lp\abs{B_K(X^6)}\frac{1}{X^{2n}}\rp^{-1} \\
		&\ll \frac{\log(X^6)}{X^{6n/2}} X^{2n} \to 0
	\end{align*}
	Therefore, condition (1) of \Cref{lem:akt19} is automatically satisfied for \( 100\% \) of elliptic curves over any number field \( K \). When \( K \) is Galois, \Cref{lem:an20} immediately yields that the same curves satisfying condition (1) also satisfies condition (3). Therefore, by \Cref{lem:akt19}, for a Galois CM field \( K \) with \( \zeta_5\notin K \), the limit that is calculated in \Cref{thm:main} depends only on the proportion of elliptic curves \( E \) over \( K \) such that \( E_{K_\mf{p}} \) is ordinary for each place \( \mf{p}\mid 5 \).
	
	Let \( E:y^2 = x^3 + Ax+B \) with \( (A,B)\in\mc{O}_K^2 \). Let \( 5\mc{O}_K = (\mf{p}_1\cdots\mf{p}_r)^e \) be the unique factorization. For \( E \) to have good reduction at \( \mf{p} \) for each \( \mf{p}\mid 5 \) it is sufficient that \( \Delta(A,B) = -16(4A^3 + 27B^2)\neq 0\pmod {\mf{p}_i} \) for all \( 1\leq i \leq r \). Note that this is not a necessary condition because the equation for \( E \) need not be a minimal integral model at \( \mf{p}_i \) for each \( 1\leq i \leq r \). Let \( R = \mf{p}_1\cdots\mf{p}_r \) denote the radical of \( 5\mc{O}_K \).
	\[ \lp \frac{\mc{O}_K}{R} \rp^2 = \lp\frac{\mc{O}_K}{\mf{p}_1\dots \mf{p}_r}\rp^2 = \lp\frac{\mc{O}_K}{\mf{p}_1}\times\dots\times\frac{\mc{O}_K}{\mf{p}_r}\rp^2 = \lp\frac{\mc{O}_K}{\mf{p}_1}\rp^2 \times\dots\times \lp\frac{\mc{O}_K}{\mf{p}_r}\rp^2 = \bb{F}_{5^f}^2\times\dots\times \bb{F}_{5^f}^2 \]
	Working over the finite field \( \bb{F}_{5^f} \), the equation \( \Delta(A,B) = 0 \) defines a singular elliptic curve over \( \bb{F}_{5^f} \) with a single cusp at the origin. By Exercise 3.5 of \cite{S09}, \( \Delta(A,B)_{\text{ns}}(\bb{F}_{5^f}) \cong \bb{F}_{5^f} \) as additive groups. Note that \( \Delta(A,B)_{\text{ns}}(\bb{F}_{5^f}) \) includes the point at infinity, and does not include the singular point \( (0,0) \). It follows that the number of pairs in \( (A,B)\in\bb{F}_{5^f}^2 \) such that \( \Delta(A,B) = 0 \) over \( \bb{F}_{5^f} \) is \( 5^f-1+1 = 5^f \) and hence \( (5^f)^2- 5^f \) pairs in \( \bb{F}_{5^f}^2 \) describe non-singular curves over \( \bb{F}_{5^f} \). Therefore, the number of classes in \( (\mc{O}_K/R)^2 \) describing elliptic curves over \( K \) with good reduction at every \( \mf{p}\mid 5 \) is at least
	\[ ((5^f)^2 - 5^f)^r = (5^f(5^f-1))^r \]
	Next, we use \Cref{lem:s09} to find out which of these elliptic curves with good reduction are also ordinary.
	\[ (x^3 +Ax+B)^2 = x^6 + 2Ax^4 + 2Bx^3 + A^2x^2 + 2ABx + B^2 \]
	Again, let us first work over the finite field \( \bb{F}_{5^f} \). Suppose \( \Delta(A,B)\neq 0 \). Looking at the \( x^4 \) coefficient in the above expansion, we see that \( E \) is ordinary if and only if \( A\neq 0 \). Consequently, \( E \) is supersingular if and only if \( A = 0 \) whence necessarily \( B\neq 0 \). Therefore, out of the \( (5^f)^2 - 5^f \) pairs in \( \bb{F}_{5^f}^2 \) describing non-singular curves over \( \bb{F}_{5^f} \), at most \( 5^f-1 \) of them correspond to supersingular curves. It follows that at least \( ((5^f)^2-5^f)-(5^f-1) = (5^f-1)^2 \) pairs in \( \bb{F}_{5^f}^2 \) describe non-singular curves over \( \bb{F}_{5^f} \) which are also ordinary. Therefore, the number of classes in \( (\mc{O}_K/R)^2 \) describing elliptic curves over \( K \) with good and ordinary reduction at every \( \mf{p}\mid 5 \) is at least
	\[ (5^f-1)^{2r} \]
	Finally, just divide by the number of elements in \( (\mc{O}_K/R)^2 \).
	\[ \frac{(5^f-1)^{2r}}{(5^{f})^{2r}} = \lp 1 - \frac{1}{5^f}\rp^{2r} \qedhere \]
\end{proof}

\section{Applications}

\begin{cor}
	Let \( K \) be a Galois CM field with \( \zeta_5\notin K \). Then
	\[ M_K \geq (4/5)^{2[K:\bb{Q}]} \]
\end{cor}
\begin{cor}
	Let \( K = \bb{Q}(\sqrt{-n}) \) with \( n > 0 \) be an imaginary quadratic field. Then
	\begin{enumerate}[(1)]
		\item \( M_K \geq (4/5)^4 \geq 0.409 \) if \( 5 \) splits in \( K \).
		\item \( M_K \geq (24/25)^2 \geq 0.921 \) if \( 5 \) is inert in \( K \).
		\item \( M_K \geq (4/5)^2 = 0.64 \) if \( 5 \) ramifies in \( K \).
	\end{enumerate}
\end{cor}
\begin{cor}
	Let \( K = \bb{Q}(\zeta_n) \) with \( n\geq 2 \) be a cyclotomic field and \( 5\nmid n \). Then for every \( \eps > 0 \) there exists \( N \) such that for all \( n\geq N \), \( M_K\geq 1-\eps \).
\end{cor}
\begin{proof}
	Indeed, \( f \) is the multiplicative order of \( 5\pmod n \). Trivially, \( 5^f \geq n+1 \) and \( f\geq \log_5(n+1) \). Therefore, \( r= \varphi(n)/f \leq n/\log_5(n+1) \).
	\[ \lp 1 - \frac{1}{5^f}\rp^{2r}\geq \lp 1 - \frac{1}{n+1}\rp^{2n/\log_5(n+1)} \]
	This function is increasing, and tends to \( 1 \) as \( n \to \infty \).
\end{proof}

\section{Acknowledgements}

I would like to thank my master's thesis advisor Patrick Allen for suggesting to me this problem and for his overall guidance as well as many helpful comments in the revision of this document.

\bibliographystyle{alpha}
\bibliography{lower-bound-cm}

\begin{thebibliography}{AKT19}

\bibitem[AKT19]{AKT19}
Patrick~B. Allen, Chandrashekhar Khare, and Jack~A. Thorne.
\newblock Modularity of $\operatorname{GL}_2(\mathbb{F}_p)$-representations
  over {CM} fields, 2019.

\bibitem[AN20]{AN20}
Patrick~B. Allen and James Newton.
\newblock Monodromy for some rank two {G}alois representations over {CM}
  fields.
\newblock {\em Doc. Math.}, 25:2487--2506, 2020.

\bibitem[CS17]{CS17}
Ana Caraiani and Peter Scholze.
\newblock On the generic part of the cohomology of compact unitary {S}himura
  varieties.
\newblock {\em Ann. of Math. (2)}, 186(3):649--766, 2017.

\bibitem[Mar18]{M18}
Daniel~A. Marcus.
\newblock {\em Number fields}.
\newblock Universitext. Springer, Cham, 2018.
\newblock Second edition of [ MR0457396], With a foreword by Barry Mazur.

\bibitem[Sil09]{S09}
Joseph~H. Silverman.
\newblock {\em The arithmetic of elliptic curves}, volume 106 of {\em Graduate
  Texts in Mathematics}.
\newblock Springer, Dordrecht, second edition, 2009.

\bibitem[Zyw10]{Z10}
David Zywina.
\newblock Elliptic curves with maximal {G}alois action on their torsion points.
\newblock {\em Bull. Lond. Math. Soc.}, 42(5):811--826, 2010.

\end{thebibliography}

\end{document}